\documentclass[12pt]{article}

\usepackage{amsmath, amssymb}
\usepackage{amsthm}
\usepackage{enumitem}

\usepackage{multirow}
\usepackage{float}

\usepackage[hidelinks]{hyperref}

\newcommand\Tstrut{\rule{0pt}{2.6ex}}         
\newcommand\Bstrut{\rule[-0.9ex]{0pt}{0pt}}   

\usepackage{color}
\definecolor{darkred}{RGB}{139,0,0}
\definecolor{darkgreen}{RGB}{0,100,0}
\definecolor{darkmagenta}{RGB}{139,0,139}
\definecolor{darkpurple}{RGB}{110,0,180}
\definecolor{darkblue}{RGB}{40,0,200}
\definecolor{darkorange}{RGB}{255,140,0}

\newcommand{\bsx}{\boldsymbol{x}}
\newcommand{\bsh}{\boldsymbol{h}}

\newcommand{\bsgamma}{\boldsymbol{\gamma}}

\newcommand{\bsk}{\boldsymbol{k}}

\newcommand{\bsy}{\boldsymbol{y}}

\newcommand{\rd}{\,{\rm d}}

\newcommand{\RR}{\mathbb{R}}
\newcommand{\NN}{\mathbb{N}}
\newcommand{\ZZ}{\mathbb{Z}}
\newcommand{\CC}{\mathbb{C}}

\newcommand{\cA}{\mathcal{A}}

\newcommand{\cO}{\mathcal{O}}
\newcommand{\cH}{\mathcal{H}}
\newcommand{\icomp}{\mathtt{i}}
\newtheorem{definition}{Definition}
\newtheorem{theorem}{Theorem}

\newtheorem{remark}[theorem]{Remark}
\newtheorem{lemma}[theorem]{Lemma}

\newcommand{\APP}{{\rm APP}}
\newcommand{\uu}{{\mathfrak u}}

\newcommand{\tra}{{\rm trace}}

\begin{document}

\title{Tractability of approximation in the weighted Korobov space in the worst-case setting --- a complete picture}

\author{Adrian Ebert and Friedrich Pillichshammer\thanks{The authors are supported by the Austrian Science Fund (FWF),
Projects F5506-N26 (Ebert) and F5509-N26 (Pillichshammer), which are parts of the Special Research Program ``Quasi-Monte Carlo Methods: Theory and Applications''.}}

\date{}

\maketitle

\begin{abstract}
In this paper, we study tractability of $L_2$-approximation of one-periodic functions from weighted Korobov spaces in the worst-case setting. The considered weights are of product form. For the algorithms we allow information from the class $\Lambda^{{\rm all}}$ consisting of all continuous linear functionals and from the class $\Lambda^{{\rm std}}$, which only consists of function evaluations. 

We provide necessary and sufficient conditions on the weights of the function space for quasi-polynomial tractability, uniform weak tractability, weak tractability and $(\sigma,\tau)$-weak tractability. Together with the already known results for strong polynomial and polynomial tractability, our findings provide a complete picture of the weight conditions for all current standard notions of tractability.
\end{abstract}


\section{Introduction}

We study tractability of $L_2$-approximation of multivariate one-periodic functions from weighted Korobov spaces of finite smoothness $\alpha$ in the worst-case setting. The considered weights are of product form. This problem has already been studied in a vast number of articles and a lot is known for the two information classes $\Lambda^{{\rm all}}$ and $\Lambda^{{\rm std}}$, in particular for the primary notions of strong polynomial and polynomial tractability, but also for weak tractability; see, e.g.,~\cite{KSW06,NSW04,WW99,WW01} and also the books \cite{NW08,NW12}. However, there are also some newer tractability notions such as quasi-polynomial tractability (see~\cite{GW11}), $(\sigma,\tau)$-weak tractability (see~\cite{SW15}) or uniform weak tractability (see~\cite{S13}) which have not yet been considered for the approximation problem for weighted Korobov spaces. Indeed, in \cite[Open Problem~103]{NW12} Novak and Wo\'{z}niakowski asked for appropriate weight conditions that characterize quasi-polynomial tractability.

It is the aim of the present paper to close this gap and to provide matching necessary and sufficient conditions for quasi-polynomial, $(\sigma,\tau)$-weak and uniform weak tractability for both information classes $\Lambda^{{\rm all}}$ and $\Lambda^{{\rm std}}$, and therefore to extend and complete the already known picture regarding tractability of $L_2$-approximation in weighted Korobov spaces. In particular, we show that for the information class $\Lambda^{{\rm all}}$ the notions of quasi-polynomial tractability, uniform weak tractability and weak tractability are equivalent and any of these holds if and only if the weights become eventually less than one (see Theorem~\ref{thm_all}). For the class $\Lambda^{{\rm std}}$ we show that polynomial tractability and quasi-polynomial tractability are equivalent and additionally provide matching sufficient and necessary conditions for the considered notions of weak tractability (see Theorem~\ref{thm_std}). 

The remainder of this article is organized as follows. In Section~\ref{sec:basics} we recall the underlying function space setting of weighted Korobov spaces with finite smoothness and provide the basics about $L_2$-approximation for such spaces. Furthermore, we give the definitions of the considered tractability notions. The obtained results are presented in Section~\ref{sec:results}. Finally, the corresponding proofs can be found in Section~\ref{sec:proofs}. 

\section{Basic definitions} \label{sec:basics}

\subsubsection*{Function space setting}

The Korobov space $\cH_{s,\alpha,\bsgamma}$ with weight sequence $\bsgamma=(\gamma_j)_{j \ge 1} \in \RR^{\NN}$ is a reproducing kernel Hilbert space with kernel function $K_{s,\alpha,\bsgamma}: [0,1]^s \times [0,1]^s \to \CC$ given by
\begin{equation*}
	K_{s,\alpha,\bsgamma}(\bsx,\bsy)
	:=
	\sum_{\bsk \in \ZZ^s} r_{s,\alpha,\bsgamma}(\bsk) \exp(2 \pi \icomp \bsk\cdot (\bsx-\bsy))
\end{equation*}
and corresponding inner product
\begin{equation*}
	\langle f,g \rangle_{s,\alpha,\bsgamma}
	:=
	\sum_{\bsk \in \ZZ^s} \frac1{r_{s,\alpha,\bsgamma}(\bsk)} \, \widehat{f}(\bsk) \, \overline{\widehat{g}(\bsk)}
	\quad \text{and} \quad
	\|f\|_{s,\alpha,\bsgamma}
	=
	\sqrt{\langle f,f \rangle_{s,\alpha,\bsgamma}}\
	.
\end{equation*}
Here, the Fourier coefficients are given by
\begin{equation*}
	\widehat{f}(\bsk)
	=
	\int_{[0,1]^s} f(\bsx) \exp(-2\pi \icomp \bsk\cdot \bsx) \rd \bsx
\end{equation*}
and the used decay function equals $r_{s,\alpha,\bsgamma}(\bsk) = \prod_{j=1}^s r_{\alpha,\gamma_j}(k_j)$ with $\alpha > 1$ (the so-called smoothness parameter of the space) and
\begin{equation*}
	r_{\alpha,\gamma}(k)
	:=
	\left\{\begin{array}{ll}
		1 & \text{for } k=0, \\[0.5em]
		\frac{\gamma}{|k|^{\alpha}}  & \text{for } k \in \ZZ\setminus\{0\}.
	\end{array}\right.
\end{equation*}
The kernel $K_{s,\alpha,\bsgamma}$ is well defined for $\alpha > 1$ and for all $\bsx,\bsy \in [0,1]^s$, since
\begin{equation*}
	|K_{s,\alpha,\bsgamma}(\bsx,\bsy)| 
	\le 
	\sum_{\bsk \in \ZZ^s} r_{s,\alpha,\bsgamma}(\bsk) 
	= 
	\prod_{j=1}^s \left(1+2 \zeta(\alpha) \gamma_j\right) 
	< 
	\infty
	,
\end{equation*}
where $\zeta(\cdot)$ is the Riemann zeta function (note that $\alpha > 1$ and hence $\zeta(\alpha)<\infty$).

Furthermore, we assume in this article that the weights satisfy $1 \ge \gamma_1 \ge \gamma_2 \ge \dots \ge 0$.

\subsubsection*{Approximation in the weighted Korobov space}

We consider the operator $\APP_s: \cH_{s,\alpha,\bsgamma} \to L_2([0,1]^s)$ with $\APP_s (f) = f$ for all $f \in \cH_{s,\alpha,\bsgamma}$. The operator $\APP_s$ is the embedding from the weighted Korobov space $\cH_{s,\alpha,\bsgamma}$ to the space $L_2([0,1]^s)$. It is compact since $\alpha>1$; see \cite{NSW04}.
In order to approximate $\APP_s$ with respect to the $L_2$-norm $\|\cdot\|_{L_2}$ over $[0,1]^s$, it is well known (see \cite[Theorems~4.5 and 4.8]{NW08} or \cite{TWW})  that it suffices to employ linear algorithms $A_{n,s}$ 
that use $n$ information evaluations and are of the form
\begin{equation} \label{eq:alg_form}
	A_{n,s}(f)
	=
	\sum_{i=1}^n T_i(f) \, g_i
	\quad \text{for }
	f \in \cH_{s,\alpha,\bsgamma}	
\end{equation}
with functions $g_i \in L_2([0,1]^s)$ and bounded linear functionals $T_i \in \cH^\ast_{s,\alpha,\bsgamma}$ for $i = 1,\ldots,n$. We will assume that the considered functionals $T_i$ belong to some permissible class of information $\Lambda$. In particular, we study the class $\Lambda^{{\rm all}}$ consisting of the entire dual space $\cH^\ast_{s,\alpha,\bsgamma}$ and the class $\Lambda^{{\rm std}}$, which consists only of point evaluation functionals. Remember that $\cH_{s,\alpha,\bsgamma}$ is a reproducing kernel Hilbert space, which means that point evaluations are continuous linear functionals and therefore $\Lambda^{{\rm std}}$ is a subclass of $\Lambda^{{\rm all}}$. \\

The worst-case error of an algorithm $A_{n,s}$ as in \eqref{eq:alg_form} is then defined as
\begin{equation*}
	e(A_{n,s})
	:=
	\sup_{\substack{f \in \cH_{s,\alpha,\bsgamma} \\ \|f\|_{s,\alpha,\bsgamma} \le 1}} \|\APP_s (f) - A_{n,s}(f)\|_{L_2}
\end{equation*}
and the $n$-th minimal worst-case error with respect to the information class $\Lambda$ is given by
\begin{equation*}
	e(n,\APP_s;\Lambda)
	:=
	\inf_{A_{n,s} \in \Lambda} e(A_{n,s})
	,
\end{equation*}
where the infimum is extended over all linear algorithms of the form \eqref{eq:alg_form} with information from the class $\Lambda$. We are interested in how the approximation error of algorithms $A_{n,s}$ depends on the number of used information evaluations $n$ and how it depends on the problem dimension~$s$. To this end, we define the so-called information complexity as
\begin{equation*}
	n(\varepsilon,\APP_s; \Lambda)
	:=
	\min\{n \in \NN_0 \ : \  e(n,\APP_s;\Lambda) \le \varepsilon \}
\end{equation*}
with $\varepsilon \in (0,1)$ and $s \in \NN$. We note that it is well known and easy to see that the initial error equals one for the considered problem and therefore there is no need to distinguish between the normalized and the absolute error criterion. 

\subsubsection*{Notions of tractability} 

In order to characterize the dependence of the information complexity on the dimension~$s$ and the error threshold~$\varepsilon$, we will study several notions of tractability which are given in the following definition.

\begin{definition}
	Consider the approximation problem $\APP=(\APP_s)_{s \ge 1}$ for the information class $\Lambda$. We say we have:
	\begin{enumerate}[label=\rm{(\alph*)}]
		\item Polynomial tractability \textnormal{(PT)} if there exist non-negative numbers $\tau, \sigma, C$ such that
		\begin{equation*}
			n(\varepsilon,\APP_s; \Lambda)
			\le
			C \, \varepsilon^{-\tau} s^\sigma
			\quad \text{for all} \quad
			s \in \NN, \,\varepsilon \in (0,1)
			.
		\end{equation*}
		\item Strong polynomial tractability \textnormal{(SPT)} if there exist non-negative numbers $\tau, C$ such that
		\begin{equation}\label{defSPT}
			n(\varepsilon,\APP_s; \Lambda)
			\le
			C \, \varepsilon^{-\tau}
			\quad \text{for all} \quad
			s \in \NN,\, \varepsilon \in (0,1)
			.
		\end{equation}
		The infimum over all exponents $\tau \ge 0$ such that \eqref{defSPT} holds for some $C \ge 0$ is called the exponent of strong polynomial tractability and is denoted by $\tau^{\ast}(\Lambda)$.
		\item Weak tractability \textnormal{(WT)} if
		\begin{equation*}
			\lim_{s + \varepsilon^{-1} \to \infty} \frac{\ln n(\varepsilon,\APP_s; \Lambda)}{s + \varepsilon^{-1}}
			=
			0
			.
		\end{equation*}
		\item Quasi-polynomial tractability \textnormal{(QPT)} if there exist non-negative numbers $t, C$ such that
		\begin{equation}\label{defQPT}
			n(\varepsilon,\APP_s; \Lambda)
			\le
			C \, \exp(t \,(1 + \ln s) (1 + \ln \varepsilon^{-1}))
			\quad \text{for all} \quad
			s \in \NN,\, \varepsilon \in (0,1)
			.
		\end{equation} 
The infimum over all exponents $t \ge 0$ such that \eqref{defQPT} holds for some $C \ge 0$ is called the exponent
		of quasi-polynomial tractability and is denoted by $t^{\ast}(\Lambda)$.
		\item $(\sigma,\tau)$-weak tractability \textnormal{($(\sigma,\tau)$-WT)} for positive $\sigma,\tau$ if
		\begin{equation*}
			\lim_{s + \varepsilon^{-1} \to \infty} \frac{\ln n(\varepsilon,\APP_s; \Lambda)}{s^\sigma + \varepsilon^{-\tau}}
			=
			0
			.
		\end{equation*}
		\item Uniform weak tractability \textnormal{(UWT)} if $(\sigma,\tau)$-weak tractability holds for all $\sigma,\tau >0$.    
	\end{enumerate} 
\end{definition}    

We obviously have the following hierarchy of tractability notions: 
\begin{equation*}
	\text{SPT} \Rightarrow \text{PT} \Rightarrow \text{QPT} \Rightarrow \text{UWT} \Rightarrow (\sigma,\tau)\text{-WT} \quad \text{for all } (\sigma,\tau)\in (0,\infty)^2.
\end{equation*}
Furthermore, WT coincides with $(\sigma,\tau)$-WT for $(\sigma,\tau)=(1,1)$. 

For more information about tractability of multivariate problems we refer to the three volumes \cite{NW08,NW10,NW12} by Novak and Wo\'{z}niakowski.

\section{The results}\label{sec:results}

Here we state our results about quasi-polynomial-, weak- and uniform weak tractability of approximation in the weighted Korobov space $\cH_{s,\alpha,\bsgamma}$ for information from $\Lambda^{{\rm all}}$. In order to provide a complete picture of all instances at a glance, we also include the already known results for (strong) polynomial tractability which were first proved by Wasilkowski and Wo\'{z}niakowski in \cite{WW99}. In the remainder of this article, we will write $\bsgamma_I$ to denote the infimum of the sequence $\bsgamma = (\gamma_j)_{j \ge 1}$. 

\begin{theorem} \label{thm_all}
	Consider the approximation problem $\APP=(\APP_s)_{s \ge 1}$ for the information class $\Lambda^{{\rm all}}$ and let $\alpha>1$. Then we have the following conditions:
	\begin{enumerate}
		\item (Cf.~\cite{WW99}) Strong polynomial tractability for the class $\Lambda^{\mathrm{all}}$ holds if and only if $s_{\bsgamma}< \infty$, where for 
		$\bsgamma=(\gamma_j)_{j \ge 1}$ the sum exponent $s_{\bsgamma}$ is defined as
		\begin{equation*}
			s_{\bsgamma}=\inf\left\{\kappa>0 \ : \ \sum_{j=1}^{\infty} \gamma_j^{\kappa} < \infty \right\}
			,
		\end{equation*}
		with the convention that $\inf \emptyset=\infty$. In this case the exponent of strong polynomial tractability is 
		\begin{equation*}
			\tau^{\ast}(\Lambda^{\mathrm{all}})
			=
			2 \max\left(s_{\bsgamma},\frac{1}{\alpha}\right)
			.
		\end{equation*}
		\item (Cf.~\cite{WW99}) Strong polynomial tractability and polynomial tractability for the class $\Lambda^{\mathrm{all}}$ are equivalent.
		\item Quasi-polynomial tractability, uniform weak tractability and weak tractability for the class $\Lambda^{{\rm all}}$ are equivalent and hold if and only if $\bsgamma_I := \inf_{j \ge 1} \gamma_j < 1$.
		\item If we have quasi-polynomial tractability, then the exponent of quasi-polynomial tractability satisfies 
		\begin{equation*}
			t^{\ast}(\Lambda^{{\rm all}}) = 2 \max\left(\frac{1}{\alpha} , \frac{1}{\ln \bsgamma_I^{-1}}\right)
			.
		\end{equation*}
		In particular, if $\bsgamma_I=0$, then we set $(\ln \bsgamma_I^{-1})^{-1}:=0$ and
		we have that $t^{\ast}(\Lambda^{{\rm all}}) = \frac{2}{\alpha}$.
		\item For $\sigma >1$, weak $(\sigma,\tau)$-tractability for the class $\Lambda^{\rm{all}}$ holds for all weights $1 \ge \gamma_1 \ge \gamma_2 \ge \dots \ge 0$.
		\end{enumerate}
\end{theorem}

\begin{remark}\rm
	We remark that in \cite{NW08} a different formulation of the necessary and sufficient condition for weak tractability is given. In particular, according to \cite[Theorem 5.8]{NW08} the approximation problem $\APP = (\APP_s)_{s \ge 1}$ for $\Lambda^{{\rm all}}$ is weakly tractable if and only if
	\begin{equation} \label{eq:WT_alt_cond}
		\lim_{s + \varepsilon^{-1} \to \infty} \frac{k(\varepsilon,s,\bsgamma)}{s + \varepsilon^{-1}}
		=
		0
		,
	\end{equation}
	where $k(\varepsilon,s,\bsgamma)$ is defined as the element $k \in \{1,\ldots ,s\}$ such that
	\begin{equation*}
		\prod_{j=1}^{k} \gamma_j > \varepsilon^2 \quad \text{and} \quad \prod_{j=1}^{k+1} \gamma_j \le \varepsilon^2
		.
	\end{equation*}
	If such a $k$ does not exist, we set $k(\varepsilon,s,\bsgamma) = s$. In the following, we show that this condition is equivalent to our condition that $\bsgamma_I <1$.
	
	Assume that $\bsgamma_I < 1$. Hence there exists an index $j_0 \in \NN$ such that $\gamma_{j_0} =: \gamma_\ast<1$
	and we see that for $k \ge j_0$ we have
	\begin{equation*}
		\prod_{j=1}^{k+1} \gamma_j \le \prod_{j=j_0}^{k+1} \gamma_j \le \prod_{j=j_0}^{k+1} \gamma_\ast = \gamma_\ast^{k-j_0+2}
		.
	\end{equation*}
	For given $\varepsilon > 0$, denote by $k_{\ast}$ the smallest positive integer such that $\gamma_\ast^{k_*-j_0+2} \le \varepsilon^2$. Elementary transformations show that this inequality  is equivalent to 
	\begin{equation*}
		k_* \ge \frac{2 \ln \varepsilon^{-1}}{\ln \gamma_\ast^{-1}} + j_0 - 2
		,
	\end{equation*}
	where here we used that $\gamma_{\ast} <1$. This however implies that $$k(\varepsilon,s,\bsgamma) \le \left\lceil \frac{2 \ln \varepsilon^{-1}}{\ln \gamma_\ast^{-1}} + j_0 - 2\right\rceil.$$
	Therefore, we obtain that
	\begin{equation*}
		\lim_{s + \varepsilon^{-1} \to \infty} \frac{k(\varepsilon,s,\bsgamma)}{s + \varepsilon^{-1}}
		\le
		\lim_{s + \varepsilon^{-1} \to \infty} \frac{\frac{2 \ln \varepsilon^{-1}}{\ln \gamma_\ast^{-1}} + j_0 - 1}{s + \varepsilon^{-1}}
		=
		0
	\end{equation*}
	and thus the condition in \eqref{eq:WT_alt_cond} is satisfied. 

	On the other hand, assume that \eqref{eq:WT_alt_cond} is satisfied but $\gamma_j=1$ for all $j \in \NN$. Then, according to the definition we obviously have that $k(\varepsilon,s,\bsgamma) = s$ for all $\varepsilon \in (0,1)$. But then, we have for fixed $\varepsilon \in (0,1)$ that 
	\begin{equation*}
		\lim_{s \to \infty} \frac{k(\varepsilon,s,\bsgamma)}{s + \varepsilon^{-1}}
		=
		\lim_{s  \to \infty} \frac{s}{s + \varepsilon^{-1}}
		= 
		1
	\end{equation*}
	and this contradicts \eqref{eq:WT_alt_cond}. Hence the $\gamma_j$ have to become eventually less than~$1$, which implies that 
	$\bsgamma_I = \inf_{j \ge 1} \gamma_j < 1$. \qed
\end{remark}

In the next theorem we present the respective conditions for tractability of approximation in the weighted Korobov space for the information class $\Lambda^{{\rm std}}$. In order to provide a detailed overview, we also include the already known results for (strong) polynomial tractability, see, e.g., \cite{NSW04}.

\begin{theorem}\label{thm_std}
	Consider multivariate approximation $\APP = (\APP_s)_{s \ge 1}$ for the information class $\Lambda^{{\rm std}}$ and $\alpha>1$. 
	Then we have the following conditions:
	\begin{enumerate}
		\item (Cf.~\cite{NSW04}) Strong polynomial tractability for the class $\Lambda^{{\rm std}}$ holds if and only if
		\begin{equation*}
			\sum_{j =1}^{\infty} \gamma_j < \infty
		\end{equation*}
		(which implies $s_{\bsgamma} \le 1$). In this case the exponent of strong polynomial tractability satisfies 
		\begin{equation*}
			\tau^{\ast}(\Lambda^{\mathrm{std}}) =2 \max\left(s_{\bsgamma},\frac{1}{\alpha} \right)
			.
		\end{equation*}
		\item (Cf.~\cite{NSW04}) Polynomial tractability for the class $\Lambda^{{\rm std}}$ holds if and only if
		\begin{equation*}	
			\limsup_{s \to \infty} \frac1{\ln s} \sum_{j=1}^s \gamma_j < \infty
			.
		\end{equation*}
		\item Polynomial and quasi-polynomial tractability for the class $\Lambda^{{\rm std}}$ are equivalent.
		\item  Weak tractability for the class $\Lambda^{{\rm std}}$ holds if and only if
		\begin{equation}\label{cond_wt_std}
			\lim_{s \to \infty} \frac1{s} \sum_{j=1}^s \gamma_j = 0
			.
		\end{equation}
		\item For $\sigma \in (0,1]$, weak $(\sigma,\tau)$-tractability for the class $\Lambda^{\rm{std}}$ holds if and only if
		\begin{equation}\label{cond_tswt_std}
			\lim_{s \to \infty} \frac1{s^\sigma} \sum_{j=1}^s \gamma_j = 0
			.
		\end{equation}
		For $\sigma >1$, weak $(\sigma,\tau)$-tractability for the class $\Lambda^{\rm{std}}$ holds for all weights $1 \ge \gamma_1 \ge \gamma_2 \ge \dots \ge 0$.
		\item Uniform weak tractability for the class $\Lambda^{\rm{std}}$ holds if and only if
		\begin{equation}\label{cond_uwt_std}
			\lim_{s \to \infty} \frac1{s^\sigma} \sum_{j=1}^s \gamma_j = 0
			\quad \text{for all } \sigma \in (0,1]
			.
		\end{equation}
	\end{enumerate}
\end{theorem} 

The proofs of the statements in Theorems~\ref{thm_all} and \ref{thm_std} are given in the next section. \\

The results in Theorems~\ref{thm_all} and~\ref{thm_std} provide a complete characterization for tractability of approximation in the weighted Korobov space $\cH_{s,\alpha,\bsgamma}$ with respect to all commonly studied notions of tractability and the two information classes $\Lambda^{{\rm all}}$ and $\Lambda^{{\rm std}}$. We summarize the conditions in a concise table (Table \ref{tab:tract_overview}) below.

\setlength{\tabcolsep}{9pt}
\setlength{\arrayrulewidth}{0.7pt}

\begin{table}[H]
	\centering
	\begin{tabular}{c||c|c}
		& $\Lambda^{{\rm all}}$ & $\Lambda^{{\rm std}}$ \Tstrut\Bstrut \\ 
		\hline
		{\rm SPT} & $s_{\bsgamma} < \infty$ & $\sum_{j=1}^{\infty} \gamma_j < \infty$ \Tstrut\Bstrut \\[0.65em]
		{\rm PT}   & $s_{\bsgamma} < \infty$ & $\limsup_{s \rightarrow \infty} \frac{1}{\ln s}\sum_{j=1}^s \gamma_j < \infty$ \\[0.65em]
		{\rm QPT} & $ \bsgamma_I < 1$ & $\limsup_{s \rightarrow \infty} \frac{1}{\ln s}\sum_{j=1}^s \gamma_j < \infty$ \\[0.65em]
		{\rm UWT} & $ \bsgamma_I < 1$ & $\lim_{s \rightarrow \infty} \frac{1}{s^{\sigma}}\sum_{j=1}^s \gamma_j = 0 \ \forall  \sigma \in (0,1]$ \\[0.65em]
		$(\sigma,\tau)\mbox{-WT}$ for $\sigma \in (0,1]$ & $ \bsgamma_I < 1$ &  $\lim_{s \rightarrow \infty} \frac{1}{s^{\sigma}}\sum_{j=1}^s 
		\gamma_j = 0$ \\[0.65em]
		\mbox{WT} & $ \bsgamma_I < 1$ & $\lim_{s \rightarrow \infty} \frac{1}{s}\sum_{j=1}^s \gamma_j = 0$ \\[0.65em]
		$(\sigma,\tau)\mbox{-WT}$ for $\sigma>1$ & no extra condition on $\bsgamma$ &  no extra condition on $\bsgamma$ \\	\end{tabular}
	\vspace{7pt}
	\caption{\label{tab:tract_overview} Overview of the conditions for tractability of approximation in $\cH_{s,\alpha,\bsgamma}$ for product weights satisfying $1 \ge \gamma_1 \ge \gamma_2 \ge \dots \ge 0$.}
\end{table}

\section{The proofs}\label{sec:proofs}

In this section we present the proofs of Theorem~\ref{thm_all} and Theorem~\ref{thm_std}.

\subsubsection*{The information class $\Lambda^{{\rm all}}$}

It is commonly known (see \cite[Section~4.2.3]{NW08} or \cite[Chapter~4, Section~5.8]{TWW}) that the $n$-th minimal worst-case errors $e(n,\APP_s;\Lambda)$ are directly related to the eigenvalues of the self-adjoint operator 
\begin{equation*}
	W_s 
	:= 
	\APP_s^\ast \APP_s: \cH_{s,\alpha,\bsgamma} \to \cH_{s,\alpha,\bsgamma}
	.
\end{equation*}
In the following lemma, we derive the eigenpairs of the operator $W_s$. For this purpose, we define, for $\bsx \in [0,1]^s, \bsk \in \ZZ^s$, the vectors
$e_{\bsk}(\bsx) = e_{\bsk,\alpha,\bsgamma} (\bsx):= \sqrt{r_{s,\alpha,\bsgamma}(\bsk)} \, \exp(2 \pi \icomp \bsk\cdot \bsx)$. 

\begin{lemma} \label{lemma:eigenval_W_s}
The eigenpairs of the operator $W_s$ are $(r_{s,\alpha,\bsgamma}(\bsk), e_{\bsk})$ with $\bsk \in \ZZ^s$.
\end{lemma}

This result is well known; see, e.g., \cite[p.~215]{NW08}. 

In order to exploit the relationship between the eigenvalues of $W_s$ and the information complexity, we define the set
\begin{equation*}
	\cA(\varepsilon, s)
	:=
	\{ \bsk \in \ZZ^s \ : \ r_{s,\alpha,\bsgamma}(\bsk) > \varepsilon^2 \}
	.
\end{equation*}
It is commonly known (see \cite{NW08}) that then the following identity holds
\begin{equation*}
	n(\varepsilon,\APP_s; \Lambda^{\mathrm{all}})
	=
	|\cA(\varepsilon, s)|
	.
\end{equation*}
We will use this fact also in the proof of Theorem~\ref{thm_all}, which is presented below.

\begin{proof}[Proof of Theorem~\ref{thm_all}]
	We prove the necessary and sufficient conditions for each of the listed notions of tractability. Items 1 and 2 of Theorem~\ref{thm_all} are known from very general results in \cite{WW99}. Since their direct proofs are easy for the considered instance, we include the proofs for these two parts as a warm-up.
	\begin{enumerate} 
	\item In order to give a necessary and sufficient condition for strong polynomial tractability for $\Lambda^{\rm all}$, we use a criterion from \cite[Section~5.1]{NW08}. From \cite[Theorem 5.2]{NW08} we find that the problem $\APP$ is strongly polynomially tractable for 
	$\Lambda^{\rm all}$ if and only if there exists a $\tau>0$ such that 
	\begin{equation}\label{critNW08} 
		\sup_{s \in \NN} \left(\sum_{\bsk \in \ZZ^s} (r_{s,\alpha,\bsgamma}(\bsk))^\tau \right)^{1/\tau} < \infty
	\end{equation}
	and then $\tau^\ast(\Lambda^{\mathrm{all}})=\inf\{2 \tau \ : \ \tau \text{ satisfies \eqref{critNW08}}\}.$
	
	Assume that $s_{\bsgamma} < \infty$. Then take $\tau$ such that $\tau > \max(s_{\bsgamma},\tfrac{1}{\alpha})$ and thus $\sum_{j=1}^\infty \gamma_j^\tau$ is finite. For the sum in \eqref{critNW08} we then obtain
	\begin{align}\label{su_zeile1}
		 \sum_{\bsk \in \ZZ^s} (r_{s,\alpha,\bsgamma}(\bsk))^{\tau}
		 &= 
		 \prod_{j=1}^{s}\left(\sum_{k=-\infty}^\infty (r_{\alpha,\gamma_j}(k))^\tau \right)\nonumber\\
		 &= 
		 \prod_{j=1}^{s}\left( 1 + 2 \gamma_j^\tau \sum_{k=1}^\infty \frac{1}{k^{\alpha \tau}} \right)\nonumber  \\
		 &= 
		 \prod_{j=1}^{s}\left(1 +  2 \zeta(\alpha \tau)\gamma_j^\tau \right)\\
		 &\le 
		 \exp\left( 2 \zeta(\alpha \tau) \sum_{j=1}^\infty \gamma_j^\tau\right)
		 < 
		 \infty
		 ,\nonumber
	\end{align}
	where we also used that $\tau > 1/\alpha$ and hence $\zeta(\alpha \tau)<\infty$. This implies that we have strong polynomial tractability and that 
	\begin{equation} \label{tstup}
		\tau^\ast(\Lambda^{\mathrm{all}}) \le 2 \max(s_{\bsgamma},\tfrac{1}{\alpha})
		.
	\end{equation}
	
	On the other hand, assume we have strong polynomial tractability. Then there exists a finite $\tau$ such that \eqref{critNW08} holds true. From \eqref{su_zeile1} we see that we obviously require that $\tau > \tfrac{1}{\alpha}$. Then, again using \eqref{su_zeile1}, we obtain that
	\begin{equation*}
		\sum_{\bsk \in \ZZ^s} (r_{s,\alpha,\bsgamma}(\bsk))^\tau 
		=
		\prod_{j=1}^s (1 + 2 \zeta(\alpha \tau) \gamma_j^\tau) 
		\ge
		2 \zeta(\alpha \tau) \sum_{j=1}^s \gamma_j^\tau
		.
	\end{equation*}
	Again, since \eqref{critNW08} holds true, we require that $\sum_{j=1}^\infty \gamma_j^\tau<\infty$ and hence $s_{\bsgamma}\le \tau < \infty$. 
	Combining both results yields that $\tau \ge \max(s_{\bsgamma},\tfrac{1}{\alpha})$ and hence also 
	\begin{equation}\label{tstdn}
		\tau^\ast(\Lambda^{\mathrm{all}}) 
		\ge 
		2 \max(s_{\bsgamma},\tfrac{1}{\alpha})
		.
	\end{equation}
	Equations \eqref{tstup} and \eqref{tstdn} then imply that $\tau^\ast(\Lambda^{\mathrm{all}}) = 2 \max(s_{\bsgamma},\tfrac{1}{\alpha})$.
	
	\item We use ideas from \cite{WW99}. In order to prove the equivalence of strong polynomial tractability and polynomial tractability it suffices to prove that polynomial tractability implies strong polynomial tractability. So let us assume that $\APP$ is polynomially tractable, i.e., there exist numbers $C,p>0$ and $q \ge 0$ such that
	\begin{equation*}
		n(\varepsilon,\APP_s; \Lambda^{{\rm all}}) 
		\le 
		C \,s^q \, \varepsilon^{-p} 
		\quad \mbox{for all $\varepsilon \in (0,1)$ and $s \in \NN$}
		.
	\end{equation*}
	Without loss of generality we may assume that $q$ is an integer. Take $s \in \NN$ such that $s \ge q+1$ and choose vectors $\bsk \in \ZZ^s$ with $s-q-1$ components equal to $0$ and $q+1$ components equal to $1$. The total number of such vectors is ${s \choose q+1}$. Now choose $\varepsilon_*=\frac{1}{2} \gamma_s^{(q+1)/2}$. Assume that $\bsk \in \ZZ^s$ is of the form as mentioned above and denote by 
	$\uu \subseteq \{1,\ldots,s\}$ the set of indices of $\bsk$ which are equal to $1$. Then we have 
	\begin{equation*}
		r_{s,\alpha,\bsgamma}(\bsk) 
		=
		\prod_{j \in \uu} \gamma_j \ge \gamma_s^{q+1} 
		> 
		\varepsilon_*^2
		.
	\end{equation*}
	Hence all the ${s \choose q+1}$ vectors $\bsk$ of the form mentioned above belong to $\cA(\varepsilon_*,s)$ and this implies that
	\begin{equation*}
		|\cA(\varepsilon_*,s)| 
		\ge 
		{s \choose q+1} 
		\ge 
		\frac{(s-q)^{q+1}}{(q+1)!} 
		\ge 
		\frac{s^{q+1}}{(q+1)! (q+1)^{q+1}}
		=: 
		s^{q+1} c_q
		.
	\end{equation*}
	This now yields
	\begin{equation*}
		s^{q+1} c_q \le |\cA(\varepsilon_*,s)| 
		= 
		n(\varepsilon_*,\APP_s; \Lambda^{{\rm all}})
		\le 
		C \,s^q\, \varepsilon_*^{-p} 
		= 
		2^p\, C\, s^q\, \gamma_s^{-(q+1)p/2}
		, 
	\end{equation*}
	which in turn implies that there exists a positive number $\widetilde{c}_{p,q}$ such that
	\begin{equation*}
		\gamma_s \le \frac{\widetilde{c}_{p,q}}{s^{2/((q+1)p)}}
		.
	\end{equation*}
	This estimate holds for all $s \ge q+1$. Hence the sum exponent $s_{\bsgamma}$ of the sequence $\bsgamma=(\gamma_j)_{j \ge 1}$ is finite, $s_{\bsgamma} < \infty$, and this implies by the first statement that we have strong polynomial tractability.
	
	\item We use the following criterion for QPT, taken from \cite[Theorem~23.2]{NW12} (see also \cite{KW19}), which states that QPT holds if and only if there exists a $\tau>0$ such that
	\begin{equation} \label{condQPT}
		C
		:=
		\sup_{s \in \NN} \frac{1}{s^2} \left(\sum_{j=1}^{\infty} \lambda_{s,j}^{\tau(1+\ln s)} \right)^{1/\tau} 
		< 
		\infty,
	\end{equation}
	where $\lambda_{s,j}$ is the $j$-th eigenvalue of the operator $W_s$  in non-increasing order. 
	
	Assume that $\bsgamma_I <1$. For the weighted Korobov space $\cH_{s,\alpha,\bsgamma}$ we have by Lemma~\ref{lemma:eigenval_W_s} that
	\begin{align*}
		\sum_{j=1}^{\infty} \lambda_{s,j}^{\tau(1+\ln s)}
		&=
		\sum_{\bsk \in \ZZ^s} (r_{s,\alpha,\bsgamma}(\bsk))^{\tau (1 + \ln s)}
		\\
		&=
		\prod_{j=1}^s \left( 1 + 2 \sum_{k=1}^\infty (r_{\alpha,\gamma_j}(k))^{\tau (1 + \ln s)} \right)
		\\
		&=
		\prod_{j=1}^s \left( 1 + 2 \zeta(\alpha \tau (1 + \ln s)) \gamma_j^{\tau (1 + \ln s)}\right)
		.
	\end{align*}
	In order that $\zeta_s:=\zeta(\alpha \tau (1 + \ln s)) < \infty$ for all $s \in \NN$, we need to require from now on that $\tau>1/\alpha$.
	Furthermore, we have that
	\begin{align*}
		\frac{1}{s^2} \left(\sum_{j=1}^{\infty} \lambda_{s,j}^{\tau(1+\ln s)} \right)^{1/\tau}
		&=
		\frac{1}{s^2} \left( \prod_{j=1}^s \left( 1 + 2 \zeta_s \gamma_j^{\tau (1 + \ln s)} \right) \right)^{1/\tau} \\
		&=
		\exp\left( \frac1{\tau} \sum_{j=1}^s \ln\left( 1 +2 \zeta_s \gamma_j^{\tau (1 + \ln s)} \right) - 2 \ln s \right) \\
	 	&\le
	 	\exp\left( \frac1{\tau} \, 2 \zeta_s \sum_{j=1}^s \gamma_j^{\tau (1 + \ln s)} - 2 \ln s \right)
	 	,
	\end{align*}
	where we used that $\ln(1+x) \le x$ for all $x \ge 0$. Now we use the well-known fact that $\zeta(x) \le 1 + \frac1{x-1}$ for all $x>1$ and thus
	\begin{equation*}
		\zeta_s \le 1 + \frac1{(\alpha \tau - 1) + \alpha \tau \ln s}
		.
	\end{equation*}
	Then we obtain
	\begin{align*}
		\frac{1}{s^2} \left(\sum_{j=1}^{\infty} \lambda_{s,j}^{\tau(1+\ln s)} \right)^{1/\tau} \!\!\! 
		&\le
		\exp\left( \frac2{\tau} \, \left(1 + \frac1{(\alpha \tau - 1) + \alpha \tau \ln s} \right) \sum_{j=1}^s \gamma_j^{\tau (1 + \ln s)} - 2 \ln s \right)
		.
	\end{align*}
	Next, we consider two cases:
	\begin{itemize} 
	\item Case $\bsgamma_I=0$: Then $\lim_{j \to \infty}\gamma_j=0$ and hence for every $\varepsilon>0$ there exists a positive integer $J=J(\varepsilon)$ such that $\gamma_J \le \varepsilon$. Then, we have that 
	\begin{eqnarray*}
	 \sum_{j=1}^s \gamma_j^{\tau (1 + \ln s)} & \le & \sum_{j=1}^{J-1} 1+ \sum_{j =J}^s \varepsilon^{\tau \ln s} \le J-1+ s^{1-\tau \ln \varepsilon^{-1}}
	\end{eqnarray*}
	such that choosing $\varepsilon = \exp(-1/\tau)$ yields that
	\begin{equation*}
		\sum_{j=1}^s \gamma_j^{\tau (1 + \ln s)} 
		\le 
		J
		.
	\end{equation*}
	Note that for the chosen $\varepsilon$ the integer $J$ depends on $\tau$, but it is finite for every fixed $\tau$.  Thus, if $\tau > 1/\alpha$ and $\lim_{j \rightarrow \infty} \gamma_j =0$ we have 
	\begin{align*}
		\frac{1}{s^2} \left(\sum_{j=1}^{\infty} \lambda_{s,j}^{\tau(1+\ln s)} \right)^{1/\tau}
		&\le
		\exp\left( \frac2{\tau} \left(1 +  \frac1{(\alpha \tau - 1) + \alpha \tau \ln s} \right)  J - 2 \ln s \right)
		\\
		&=
		\exp(\cO(1)) 
		< 
		\infty
		,
	\end{align*}
	for all $s \in \NN$. By the characterization in \eqref{condQPT}, this implies quasi-polynomial tractability.
    \item Case $\bsgamma_I \in (0,1)$: Then, for every $\gamma_{\ast} \in (\bsgamma_I,1)$ there exists a $j_0=j_0(\gamma_{\ast}) \in \NN$ such that 
    \begin{equation*}
    	\gamma_j \le \gamma_{\ast} 
    	< 
    	1 
    	\quad \mbox{for all}\  j > j_0.
    \end{equation*}
    Hence, we obtain for every $s \in \NN$ that 
	\begin{align*}
		\sum_{j=1}^s \gamma_j^{\tau(1+\ln s)} 
		&\le
		j_0 +  \gamma_\ast^{\tau(1+ \ln s)} \max(s-j_0,0) \\
		&=
		j_0 +\frac{\gamma_\ast^{\tau} \max(s-j_0,0)}{s^{\tau \ln \gamma_\ast^{-1}}}
		\le 
		j_0+1
		,
	\end{align*}
	as long as $\tau \ge (\ln \gamma_\ast^{-1})^{-1}$. Thus, if $\tau > 1/\alpha$ and $\tau \ge (\ln \gamma_\ast^{-1})^{-1}$, then we have 
	\begin{align*}
		\frac{1}{s^2} \left(\sum_{j=1}^{\infty} \lambda_{s,j}^{\tau(1+\ln s)} \right)^{1/\tau}
		&\le
		\exp\left( \frac2{\tau} \left(1 +  \frac1{(\alpha \tau - 1) + \alpha \tau \ln s} \right)  (j_0+1) - 2 \ln s \right)
		\\
		&=
		\exp(\cO(1)) 
		< 
		\infty
		,
	\end{align*}
	for all $s \in \NN$. Again, by the characterization in \eqref{condQPT}, this implies quasi-polynomial tractability. 
    \end{itemize}
    
	Of course, quasi-polynomial tractability implies uniform weak tractability, which in turn implies weak tractability. 
	
	It remains to show that weak tractability implies $\bsgamma_I < 1$. Assume on the contrary that $\bsgamma_I=1$, i.e., $\gamma_j=1$ for all $j \in \NN$. Then we have for all $\bsk \in \{-1,0,1\}^s$ that $r_{s,\alpha,\bsgamma}(\bsk)=1$. This yields that for all $\varepsilon \in (0,1)$ we have $\{-1,0,1\}^s \subseteq \mathcal{A}(\varepsilon,s)$ and hence $n(\varepsilon,\APP_s; \Lambda^{{\rm all}}) \ge 3^s$. This means that the approximation problem suffers from the curse of dimensionality and, in particular, we cannot have weak tractability. This concludes the proof of item 3.
	
	\item Again from \cite[Theorem~23.2]{NW12} we know that the exponent of quasi-polynomial tractability is 
	\begin{equation*}
		t^{\ast}(\Lambda^{{\rm all}})
		=
		2 \inf\{\tau \ : \ \tau \text{ for which \eqref{condQPT} holds} \}
		.
	\end{equation*}
	From the above part of the proof it follows that $\tau$ satisfies \eqref{condQPT} as long as $\tau > 1/\alpha$ and $\tau > (\ln \bsgamma_I^{-1})^{-1}$, where we put $(\ln \bsgamma_I^{-1})^{-1}:=0$ whenever $\bsgamma_I=0$. Therefore, 
	\begin{equation*}
		t^{\ast}(\Lambda^{{\rm all}})
		\le 
		2 \max\left(\frac{1}{\alpha} ,\frac{1}{\ln \bsgamma_I^{-1}}\right)
		.
	\end{equation*}
	Assume now that we have quasi-polynomial tractability. Then \eqref{condQPT} holds true for some $\tau>0$. 
	Considering the special instance $s=1$, this means that
	\begin{equation*}
		C
		\ge 
		\left(\sum_{j=1}^{\infty} \lambda_{1,j}^{\tau} \right)^{1/\tau}= \left(1+2 \zeta(\alpha \tau) \gamma_1^{\tau} \right)^{1/\tau} 
	\end{equation*}
	and hence we must have $\tau>1/\alpha$. This already implies the result $t^{\ast}(\Lambda^{{\rm all}}) = \frac{2}{\alpha}$ whenever $\bsgamma_I=0$.
	
	It remains to study the case $\bsgamma_I>0$. Now, again according to \eqref{condQPT}, there exists a $\tau>1/\alpha$ such that for all $s \in \NN$ we have 
	\begin{align*}
		C  
		&\ge
		\frac{1}{s^2} \left( \prod_{j=1}^s \left(1+2 \zeta(\alpha\tau(1+\ln s))\gamma_j^{\tau(1+\ln s)} \right)\right)^{1/\tau} \\
		&\ge
		\exp\left(\frac{1}{\tau} \sum_{j=1}^s \ln\left(1 +\gamma_j^{\tau(1+\ln s)}\right) - 2 \ln s \right)
		.
	\end{align*}
	Taking the logarithm leads to 
	\begin{align*}
		\ln C 
		&\ge
		\frac{1}{\tau} \sum_{j=1}^s \ln\left(1+\gamma_j^{\tau(1+\ln s)}\right) - 2 \ln s
		\\
		&\ge
		\frac{s}{\tau} \ln\left(1+\bsgamma_I^{\tau(1+\ln s)}\right) - 2 \ln s
	\end{align*}
	for all $s \in \NN$. Since $\bsgamma_I \in (0,1)$ and since $\ln(1+x)\ge x \ln 2$ for all $x \in [0,1]$, it follows that for all $s \in \NN$ we have 
	\begin{equation*}
		\ln C 
		\ge
		 \frac{s \ln 2}{\tau} \bsgamma_I^{\tau(1+\ln s)} -2 \ln s 
		= 
		\frac{\bsgamma_I^{\tau} s \ln 2}{\tau \, s^{\tau \ln \bsgamma_I^{-1}}} - 2 \ln s
		.
	\end{equation*}
	This implies that $\tau \ge (\ln \bsgamma_I^{-1})^{-1}$. Therefore, we also have that 
	\begin{equation*}
		t^{\ast}(\Lambda^{{\rm all}})
		\ge 
		2 \max\left(\frac{1}{\alpha} , \frac{1}{\ln \bsgamma_I^{-1}}\right)
	\end{equation*}
	and the claimed result follows.
	\item The result for $(\sigma,\tau)$-weak tractability for $\sigma>1$ for the class $\Lambda^{{\rm all}}$ follows from the corresponding result for the class $\Lambda^{{\rm std}}$ from Theorem~\ref{thm_std}.
	 \qedhere
	\end{enumerate}
\end{proof}

\subsubsection*{The information class $\Lambda^{{\rm std}}$}

Below, we provide the remaining proof of Theorem~\ref{thm_std}.

\begin{proof}[Proof of Theorem~\ref{thm_std}]
	The necessary and sufficient conditions for polynomial and strong polynomial tractability (items~1~and~2) have already been proved in \cite{NSW04}.
	See also \cite[p.~215ff.]{NW08}, where the exact exponent of strong polynomial tractability $\tau^{\ast}(\Lambda^{\mathrm{std}})$ is given. We will therefore only provide proofs for items 3 to 6. 
	
	We start with a preliminary remark about the relation between integration and approximation.  It is well known that multivariate approximation is not easier than multivariate integration ${\rm INT}_s(f)=\int_{[0,1]^s} f(\bsx) \rd \bsx$ for $f \in \cH_{s,\alpha,\bsgamma}$, see, e.g., \cite{NSW04}. In particular, necessary conditions for some notion of tractability for the integration problem are also necessary for the approximation problem. We will use this basic observation later on. Now we present the proof of item 3.
	\begin{enumerate}
	\item[3.] Obviously, it suffices to prove that quasi-polynomial tractability implies polynomial tractability. Assume therefore that quasi-polynomial tractability for the class $\Lambda^{{\rm std}}$ holds for approximation. Then we also have  quasi-polynomial tractability for the integration problem. Now we apply \cite[Theorem~16.16]{NW10} which states that integration is $T$-tractable if and only if 
	\begin{equation}\label{crit_Ttract}	
		\limsup_{s+\varepsilon^{-1} \to \infty} \frac{\sum_{j=1}^{s} \gamma_j +\ln \varepsilon^{-1}}{1+\ln T(\varepsilon^{-1},s)} < \infty
		.
	\end{equation}	
	We do not require the definition of $T$-tractability here (see, e.g., \cite[p.~291]{NW08}). For our purpose it suffices to know that the special case
	\begin{equation*}
		T(\varepsilon^{-1},s)=\exp((1+\ln s)(1+\ln \varepsilon^{-1}))
	\end{equation*}
	corresponds to quasi-polynomial tractability. For this instance condition \eqref{crit_Ttract} becomes 
	\begin{equation*}
		\limsup_{s+\varepsilon^{-1} \to \infty} \frac{\sum_{j=1}^{s} \gamma_j +\ln \varepsilon^{-1}}{1+(1+\ln s)(1+\ln \varepsilon^{-1})} 
		< 
		\infty
		.
	\end{equation*}
	Hence, setting $\varepsilon=1$ and letting $s \rightarrow \infty$, we obtain
	\begin{equation} \label{eq:cond_PT}
		\limsup_{s \rightarrow \infty} \frac{1}{\ln s} \sum_{j=1}^s \gamma_j < \infty
		.
	\end{equation}
	From item 2, we know that condition~\eqref{eq:cond_PT} implies polynomial tractability and this completes the proof of item 3.
	\end{enumerate}
	
	For the remaining conditions in items $4$ to $6$, note that since $\alpha>1$ the trace of $W_s$, denoted by $\tra(W_s)$, is finite for all $s \in \NN$. Indeed,  we have 
	\begin{equation}\label{trace_Ws}
		\tra(W_s) = \sum_{\bsk \in \ZZ^s} r_{s,\alpha,\bsgamma}(\bsk) = \prod_{j=1}^s (1+2 \gamma_j \zeta(\alpha)) < \infty
		.
	\end{equation}
	In this case, we can use relations between notions of tractability for $\Lambda^{{\rm all}}$ and $\Lambda^{{\rm std}}$ which were first proved in \cite{WW01} (see also \cite[Section~26.4.1]{NW12}). 
	\begin{enumerate}
		\item[4.-6.] We prove the three statements in one combined argument. If any of the three conditions \eqref{cond_wt_std}, \eqref{cond_tswt_std} for $\sigma \le 1$ or \eqref{cond_uwt_std} holds, then this implies that the weights $(\gamma_j)_{j\ge1}$ have to become eventually less than $1$ since otherwise, for every $\sigma \in (0,1]$,
		\begin{equation*}
			\lim_{s \to \infty} \frac1{s^\sigma} \sum_{j=1}^s \gamma_j = \lim_{s \to \infty} \frac{s}{s^\sigma} = \lim_{s \to \infty} s^{1-\sigma} \ge 1
			.
		\end{equation*}
		Therefore, we have by Theorem \ref{thm_all} that uniform weak tractability (and even quasi-polynomial tractability) holds for the class $\Lambda^{\text{all}}$. Furthermore, from \eqref{trace_Ws} we obtain 
		\begin{align*}
			\frac{\ln(\tra(W_s))}{s^\sigma} 
			&=
			\frac1{s^\sigma} \ln \left( \prod_{j=1}^{s}\left(1 + 2 \gamma_j \zeta(\alpha) \right) \right) 
			\\
			&=
			\frac1{s^\sigma} \sum_{j=1}^s \ln(1 + 2 \gamma_j \zeta(\alpha))
			\le
			\frac{2 \zeta(\alpha)}{s^\sigma} \sum_{j=1}^s \gamma_j
			,
		\end{align*}   
		and thus if $\frac1{s^\sigma} \sum_{j=1}^s \gamma_j$ converges to $0$ as $s$ goes to infinity, with $\sigma \in (0,1]$, then
		\begin{equation*}
			\lim_{s \to \infty} \frac{\ln(\tra(W_s))}{s^\sigma}
			\le
			\lim_{s \to \infty} \frac{2 \zeta(\alpha)}{s^\sigma} \sum_{j=1}^s \gamma_j
			=
			0
			.
		\end{equation*}
		By the same argument as in the proof of \cite[Theorem 26.11]{NW12}, we obtain that \eqref{cond_wt_std} implies weak tractability for the class $\Lambda^{\text{std}}$. The proof for the other two notions of weak tractability can be obtained analogously by appropriately modifying the argument used in the proof of \cite[Theorem 26.11]{NW12}.
		
		For $(\sigma,\tau)$-weak tractability with $\sigma >1$ we can use well-known results from \cite{KSW06} or \cite{KLP18}. For example, from \cite[Lemma~6]{KSW06} or likewise  from \cite[Proposition~1]{KLP18} one can easily deduce that for weights satisfying $1 \ge \gamma_1\ge \gamma_2 \ge \dots \ge 0$ we have $n(\varepsilon,\APP_s;\Lambda^{{\rm std}}) \le C \, \varepsilon^{-\eta} \, K^s$ for reals $C,\eta>0$ and $K>1$, and hence $$\ln n(\varepsilon,\APP_s;\Lambda^{{\rm std}}) \le \ln C + \eta \ln \varepsilon^{-1} + s \ln K.$$ This implies $$\lim_{s+\varepsilon^{-1}\rightarrow \infty} \frac{\ln n(\varepsilon,\APP_s;\Lambda^{{\rm std}})}{s^{\sigma}+\varepsilon^{-\tau}}=0\quad \mbox{for every}\ \sigma >1$$ and hence $\APP$ is $(\sigma,\tau)$-weakly tractable for every $\sigma>1$.
		
		It remains to prove the necessary conditions for the three notions of weak tractability. From our preliminary remark we know that necessary conditions on tractability for integration are also necessary conditions for approximation. Hence it suffices to study integration ${\rm INT}_s$.
		
		Due to, e.g., \cite{W09}, we know that weak tractability of integration for $\cH_{s,\alpha,\bsgamma}$ holds if and only if
		\begin{equation*}
			\lim_{s \to \infty} \frac1{s} \sum_{j=1}^s \gamma_j = 0
		\end{equation*}
		and thus this is also a necessary condition for weak tractability of approximation.
	
		We are left to prove the necessity of the respective conditions for uniform weak tractability and $(\sigma,\tau)$-weak tractability for integration. These follow from a similar approach as used in \cite{W09} for weak tractability. We just sketch the argument which is more or less an application and combination of results from \cite{HW2001} and \cite{NW2001}. 
		
		In \cite[Theorem~4.2]{HW2001} Hickernell and Wo\'{z}niakowski showed that integration in a suitably constructed weighted Sobolev space $\cH^{{\rm Sob}}_{s,r,\widehat{\bsgamma}}$ of smoothness $r=\lceil \alpha/2\rceil$ and with product weights $\widehat{\bsgamma}=(\widehat{\gamma}_j)_{j \ge 1}$ is no harder than in the weighted Korobov space $\cH_{s,\alpha,\bsgamma}$. The weighted Sobolev space $\cH^{{\rm Sob}}_{s,r,\widehat{\bsgamma}}$ is a reproducing kernel Hilbert space whose kernel is a product of one-dimensional reproducing kernels (see \cite[Eq.~(23)]{HW2001}), the corresponding definition can be found in \cite[Eq.~(19)]{HW2001}. The product weights of the Korobov and Sobolev spaces are related by $\gamma_j=\widehat{\gamma}_j G_r$ with a multiplicative non-negative factor $G_r$. Hence, it suffices to study necessary conditions for tractability of integration in $\cH^{{\rm Sob}}_{s,r,\widehat{\bsgamma}}$. To this end we proceed as in \cite[Section~5]{HW2001}.  
		
		The univariate reproducing kernel $K_{1,\widehat{\gamma}}$ of $\cH^{{\rm Sob}}_{1,r,\widehat{\gamma}}$ (case $s=1$) can be decomposed as
		\begin{equation*}
			K_{1,\widehat{\gamma}}
			=
			R_1 +\widehat{\gamma}(R_2+R_3)
			,
		\end{equation*}
		 where each $R_j$ is a reproducing kernel of a Hilbert space $\cH(R_j)$ of univariate functions. In our specific case, we have $R_1=1$ and $\cH(R_1)={\rm span}(1)$ (cf.~\cite[p.~679]{HW2001}). It is then shown in \cite[Section~5]{HW2001} that all requirements of \cite[Theorem~4]{NW2001} are satisfied. For the involved parameter $\alpha_1$, we have $\alpha_1=\|h_{1,1}\|_{\cH(R_1)}^2=1$ 
		 (this is easily shown, since $R_1=1$). Furthermore, we have that the parameter $\alpha$ in \cite[Theorem~4]{NW2001} (not to be confused with the smoothness parameter $\alpha$ of the Korobov space)  satisfies $\alpha \in [1/2,1)$, since $h_{1,2,(0)}\not=0$ and $h_{1,2,(1)}\not =0$, as shown in \cite[p.~681]{HW2001} (where $h_{1,2,(j)}$ is called $\eta_{1,2,(j)}$ for $j \in \{0,1\}$). In order to avoid any misunderstanding, we denote the $\alpha$ in \cite[Theorem~4]{NW2001} by $\widetilde{\alpha}$ from now on. Then, we apply \cite[Theorem~4]{NW2001} and obtain for the squared $n$-th minimal integration error in the considered Sobolev space that
		\begin{equation*} 
			e^2(n,{\rm INT}_s) 
			\ge 
			\sum_{\uu \subseteq \{1,\ldots,s\}} (1-n \widetilde{\alpha}^{|\uu|})_+\, \alpha_2^{|\uu|} \prod_{j \in \uu} \widehat{\gamma}_j 
			\prod_{j \not \in \uu}(1+\widehat{\gamma}_j \alpha_3)
			,
		\end{equation*}
		where $\alpha_2, \alpha_3$ are positive numbers (cf.~\cite[p.~425]{NW2001}) and $(x)_+:=\max(x,0)$. This implies
		\begin{align*} 
			e^2(n,{\rm INT}_s) 
			&\ge
			\sum_{\uu \subseteq \{1,\ldots,s\}} (1-n \widetilde{\alpha}^{|\uu|}) \, \alpha_2^{|\uu|} \prod_{j \in \uu} \widehat{\gamma}_j
			\\
			&=
			\prod_{j=1}^s (1+\alpha_2 \widehat{\gamma}_j) - n \prod_{j=1}^s (1 + \alpha_2 \,\widetilde{\alpha}\, \widehat{\gamma}_j)
			,
		\end{align*}	  
		which in turn yields that
		\begin{equation*}
			n(\varepsilon,{\rm INT}_s) 
			\ge 
			\frac{\prod_{j=1}^s (1+\alpha_2 \widehat{\gamma}_j) - \varepsilon^2}{\prod_{j=1}^s (1 + \alpha_2 \,\widetilde{\alpha}\, \widehat{\gamma}_j)}
			.
		\end{equation*}
		Taking the logarithm, we obtain
		\begin{align*}
			\ln n(\varepsilon,{\rm INT}_s) 
			&\ge
			\ln\left(\prod_{j=1}^s (1+\alpha_2 \widehat{\gamma}_j)\right) + \ln\left(1-\frac{\varepsilon^2}{\prod_{j=1}^s (1+\alpha_2 \widehat{\gamma}_j)}\right)
			\\
			&\quad- \ln\left( \prod_{j=1}^s (1+\alpha_2 \,\widetilde{\alpha}\, \widehat{\gamma}_j)\right)
			\\
			&\ge
			\sum_{j=1}^s \ln(1+\alpha_2 \widehat{\gamma}_j) - \alpha_2 \widetilde{\alpha} \sum_{j=1}^s \widehat{\gamma}_j + \ln(1-\varepsilon^2)
			,
		\end{align*}
		where we used that $\ln (1+x) \le x$ for any $x \ge 0$. 
		
		Recall that $\widetilde{\alpha} < 1$ and set $c:=(1+\widetilde{\alpha})/2$. Then $c \in (\widetilde{\alpha},1)$ and since 
		\begin{equation*}
			\lim_{x \to0} \frac{\ln (1+x)}{x} = 1
			,
		\end{equation*}
		it follows that $\ln(1+x) \ge c x$ for sufficiently small $x>0$. 
		
		Next, assume that we have $(\sigma,\tau)$-weak tractability for integration in the considered Sobolev space. Then the weights $\widehat{\gamma}_j$ necessarily tend to zero for $j \to \infty$ (see \cite[Theorem~4, Item~4]{NW2001}). In particular, there exists an index $j_0>0$, 
		such that for all $j \ge j_0$ we have $\ln(1+\alpha_2 \widehat{\gamma}_j) \ge c \,\alpha_2\, \widehat{\gamma}_j$. Hence for $s \ge j_0$, we have
		\begin{equation*}
			\ln n(\varepsilon,{\rm INT}_s)
			\ge 
			\alpha_2(c-\widetilde{\alpha}) \sum_{j=j_0}^s \widehat{\gamma}_j + \ln(1-\varepsilon^2) + \mathcal{O}(1)
			.
		\end{equation*}
		Note that $c-\widetilde{\alpha} >0$. Since we assume $(\sigma,\tau)$-weak tractability, we have that
		\begin{align*}
			0
			=
			\lim_{s+\varepsilon^{-1} \rightarrow \infty} \frac{\ln n(\varepsilon,{\rm INT}_s)}{s^{\sigma}+\varepsilon^{-\tau}}
			\ge
			\lim_{s+\varepsilon^{-1} \rightarrow \infty} \frac{\alpha_2(c-\widetilde{\alpha}) \sum_{j=j_0}^s \widehat{\gamma}_j + \ln(1-\varepsilon^2)}{s^{\sigma}+\varepsilon^{-\tau}}
			.
		\end{align*}
		This, however, implies that 
		\begin{equation*}
			\lim_{s \rightarrow \infty} \frac{1}{s^{\sigma}} \sum_{j=1}^s \widehat{\gamma}_j=0
			,
		\end{equation*}
		and thus, since $\gamma_j$ and $\widehat{\gamma}_j$ only differ by a multiplicative factor, that 
		\begin{equation*}
			\lim_{s \rightarrow \infty} \frac{1}{s^{\sigma}} \sum_{j=1}^s \gamma_j
			=
			0
			.
		\end{equation*}
		Now the claimed results follow.
	\end{enumerate}
\end{proof}

\paragraph{Acknowledgment.} The authors are grateful to two anonymous referees for important and very useful comments on this paper.

\begin{small}
\noindent\textbf{Authors' addresses:}

\medskip
\noindent Adrian Ebert\\
Johann Radon Institute for Computational and Applied Mathematics (RICAM)\\
Austrian Academy of Sciences\\
Altenbergerstr.~69, 4040 Linz, Austria\\
E-mail: \texttt{adrian.ebert@oeaw.ac.at}

\medskip

\noindent Friedrich Pillichshammer\\
Institut f\"{u}r Finanzmathematik und Angewandte Zahlentheorie\\
Johannes Kepler Universit\"{a}t Linz\\
Altenbergerstr.~69, 4040 Linz, Austria\\
E-mail: \texttt{friedrich.pillichshammer@jku.at}
\end{small}

\end{document}